\documentclass[11pt]{amsart}

\usepackage{amssymb}

\begin{document}

\title[Approximation of holomorphic functions  in  Banach spaces]{Holomorphic approximation in Banach spaces: A survey}
\author[F. Meylan]{Francine Meylan}
\address{F. Meylan : Institut de Math\'ematiques, Universit\'e de Fribourg, 1700
Perolles, Fribourg, Switzerland} \email{francine.meylan@unifr.ch}


\thanks{Analyse Complexe/Complex Analysis\\
The  author was partially supported by Swiss NSF Grant
2100-063464.00/1}

\def\Label#1{\label{#1}}
\def\1#1{\ov{#1}}
\def\2#1{\widetilde{#1}}
\def\6#1{\mathcal{#1}}
\def\4#1{\mathbb{#1}}
\def\3#1{\widehat{#1}}

\def\C{{\4C}}
\def\R{{\4R}}

\def\Re{{\sf Re}\,}
\def\Im{{\sf Im}\,}

\numberwithin{equation}{section}
\def\s{s}
\def\k{\kappa}
\def\ov{\overline}
\def\span{\text{\rm span}}
\def\ad{\text{\rm ad }}
\def\tr{\text{\rm tr}}
\def\xo {{x_0}}
\def\Rk{\text{\rm Rk\,}}
\def\sg{\sigma}
\def \emxy{E_{(M,M')}(X,Y)}
\def \semxy{\scrE_{(M,M')}(X,Y)}
\def \jkxy {J^k(X,Y)}
\def \gkxy {G^k(X,Y)}
\def \exy {E(X,Y)}
\def \sexy{\scrE(X,Y)}
\def \hn {holomorphically nondegenerate}
\def\hyp{hypersurface}
\def\prt#1{{\partial \over\partial #1}}
\def\det{{\text{\rm det}}}
\def\wob{{w\over B(z)}}
\def\co{\chi_1}
\def\po{p_0}
\def\fb {\bar f}
\def\gb {\bar g}
\def\Fb {\ov F}
\def\Gb {\ov G}
\def\Hb {\ov H}
\def\zb {\bar z}
\def\wb {\bar w}
\def \qb {\bar Q}
\def \t {\tau}
\def\z{\chi}
\def\w{\tau}
\def\Z{\zeta}
\def\phi{\varphi}
\def\eps{\varepsilon}

\def \T {\theta}
\def \Th {\Theta}
\def \L {\Lambda}
\def\b {\beta}
\def\a {\alpha}
\def\o {\omega}
\def\l {\lambda}

\def \im{\text{\rm Im }}
\def \re{\text{\rm Re }}
\def \Char{\text{\rm Char }}
\def \supp{\text{\rm supp }}
\def \codim{\text{\rm codim }}
\def \Ht{\text{\rm ht }}
\def \Dt{\text{\rm dt }}
\def \hO{\widehat{\mathcal O}}
\def \cl{\text{\rm cl }}
\def \bR{\mathbb R}
\def \bS{\mathbb S}
\def \bK{\mathbb K}
\def \bD{\mathbb D}
\def \bC{\mathbb C}
\def \C{\mathbb C}
\def \N{\mathbb N}
\def \bL{\mathbb L}
\def \bZ{\mathbb Z}
\def \bN{\mathbb N}
\def \scrF{\mathcal F}
\def \scrK{\mathcal K}
\def \mc #1 {\mathcal {#1}}
\def \scrM{\mathcal M}
\def \cR{\mathcal R}
\def \scrJ{\mathcal J}
\def \scrA{\mathcal A}
\def \scrO{\mathcal O}
\def \scrV{\mathcal V}
\def \scrL{\mathcal L}
\def \scrE{\mathcal E}
\def \hol{\text{\rm hol}}
\def \aut{\text{\rm aut}}
\def \Aut{\text{\rm Aut}}
\def \J{\text{\rm Jac}}
\def\jet#1#2{J^{#1}_{#2}}
\def\gp#1{G^{#1}}
\def\gpo{\gp {2k_0}_0}
\def\emmp {\scrF(M,p;M',p')}
\def\rk{\text{\rm rk\,}}
\def\Orb{\text{\rm Orb\,}}
\def\Exp{\text{\rm Exp\,}}
\def\Span{\text{\rm span\,}}
\def\d{\partial}
\def\D{\3J}
\def\pr{{\rm pr}}

\def\dbl{[\hskip -1pt [}
\def\dbr{]\hskip -1pt]}

\def \CZZ {\C \dbl Z,\zeta \dbr}
\def \D{\text{\rm Der}\,}
\def \Rk{\text{\rm Rk}\,}
\def \ima{\text{\rm im}\,}
\def \I {\mathcal I}

\newtheorem{resume}{R\'esum\'e}
\newtheorem{Thmf}{Th\'eor\`eme}[section]
\newtheorem{Corf}[Thmf]{Corollaire}
\newtheorem{Thm}{Theorem}[section]
\newtheorem{Def}[Thm]{Definition}
\newtheorem{Cor}[Thm]{Corollary}
\newtheorem{Pro}[Thm]{Proposition}
\newtheorem{Lem}[Thm]{Lemma}
\newtheorem{Rem}[Thm]{Remark}
\newtheorem{Rema}[Thm]{Remarque}
\theoremstyle{definition}\newtheorem{Exa}[Thm]{Example}
\newtheorem{abs}{Abstract}

\maketitle
 \noindent {\bf Abstract.} {\rm  We give a survey about the Runge approximation problem  for a holomorphic function defined  on the unit ball of   a complex Banach space.}

\bigskip

\section{Introduction}
Given  a complex  Banach space $X$ of {\it infinite dimension}, we recall  the  following Runge  approximation problems:
  
  \begin{enumerate}
\item[(1)] { \em Given $r  \in (0,1),$  $\epsilon >0,$ and $f$ a  holomorphic function  on the open  unit ball of $X,$ is  there an entire  function $h$  satisfying  $|f-h| < \epsilon$ on the open  ball of radius $r$ centered at the origin?  See for instance \cite{Le99}}.
\end{enumerate}

\begin{enumerate}
\item[(2)] {\em Is there $r  \in (0,1)$  such that for any $\epsilon >0,$ and any  holomorphic function $f$ on the open  unit ball of $X,$ there exists a entire function $h$  satisfying  $|f-h| < \epsilon$ on the open  ball of radius $r$ centered at the origin? See for instance  \cite{Me06}}.
\end{enumerate}
 We note that problem (2) is clearly  independent of the choice of an equivalent norm defining the topology of $X,$ while  problem (1) is the analogous version of what is known to be true when $X$ is finite dimensional. The main difficulty here is due to  Riesz Theorem  asserting  that the unit ball is not relatively  compact   when $\dim X = \infty.$  In other words, we may possibly deal with holomorphic functions defined on the unit ball that are  not bounded on smaller balls. The following example shows that this is indeed  the case.

\begin{Exa}\label{1}
Consider on the Banach space $l^1(\Bbb N)= \{z: \Bbb N \longrightarrow \Bbb C, \ ||z|| = \sum _{j=1}^\infty |z(j) |< \infty ,        \}  $    the function given by      
$$ f(z)=\sum_{j=1}^\infty 2^jz(j)^j,   $$  
 and let $e_k$ be  defined by $e_k(j)=\delta_{jk}.$  One can check that $f$  is holomorphic on the unit ball, and  since  $f(\frac{2}{3}e_k) =2^k(\frac{2}{3})^k,$   $f$ is unbounded on  $\overline {B(0,\frac{2}{3})}.$
\end{Exa}

On the other hand, since  unbounded  holomorphic functions on the unit ball may occur,  we can  not replace, as in the finite dimensional case, {\it $h$ entire holomorphic function} by   { \it $h$  holomorphic polynomial}   in the setting of the Runge approximation problems. Indeed,  holomorphic polynomials are bounded on the unit ball.

The first result in the direction of problem (1) is due to L. Lempert. He proves the following theorem for the Banach space
$l^1(\Gamma)= \{z: \Gamma \longrightarrow \Bbb C, \ ||z|| = \sum _{\gamma \in \Gamma} |z(\gamma) |< \infty ,        \}$ where  $\Gamma$  is any set.
\begin{Thm}\cite{Le99}
Let $(X, \{\phi_{\alpha}\})$  be  a locally convex space  and    $M$ be a Stein manifold. Let $K \subset M$ be a compact set  that is convex with respect to the set of  holomorphic functions on $M$, and  $V \subset M$ be an open neighborhood of $K$. 
Define
$$
\Omega:=\{(m,z)\in V\times l^1(\Gamma): ||z||< R(m)\},\ \  \omega:=\{(m,z)\in \Omega: ||z||< r(m)\},
$$
where $r$ and $R$  are  positive continuous functions on $V,$  satisfying  $r<R.$

Then for  every seminorm $\phi_{\alpha}$, for every $\epsilon>0,$ for every $X$-valued  holomorphic map f defined on $\Omega$ , there exists an  $X$-valued holomorphic map g defined on $
M\times  l^1(\Gamma)$  such that 
$ \phi_{\alpha}(f-g) < \epsilon $ on $\omega|_K.$
\end{Thm}
One of Lempert's motivations to get interested into  the Runge approximation problems has been  for instance the solvability of the $\bar{\partial}$ equation  in Banach spaces. See  \cite{Le99b}, \cite{Le00a}, \cite{Pa00}. In particular, he obtains the following theorem:
\begin{Thm}\cite{Le00a} If  $\Omega \subset l^1(\Bbb N)$ is pseudoconvex and $f \in C_{0,1}(\Omega)$ is a closed locally Lipschitz continuous $(0,1)$ form, then the equation $\bar{\partial}u=f$ has a solution $u \in C^1(\Omega).$
\end{Thm}
When solving the problem  when $\Omega$ is  the ball  centered at $0$ of radius $R$ in \cite{Le99b}, he shows first that $\bar{\partial}u=f$  is  solvable on  balls  centered at $0$ of radius $r<R,$ and then constructs a global solution on the ball  centered at $0$ of radius $R$  using the Runge approximation problem (1).

\section{Preliminaries}\Label{trick}

Let $X$ be a  complex Banach space.  For $r>0,$ we use    $B(0,r)$ to denote  the ball of radius $r.$  Recall  that a  function $f:U\subset X \longrightarrow  \mathbb C,$ where $U$ is an open subset of $X,$  is  {\it holomorphic} if $f$ is continuous on $U,$ and $f_{|U\cap X_1}$ is holomorphic,  in the classical sense, as a function of several complex variables, for each finite dimensional  subspace $X_1$ of $X.$ (See \cite{D99}.) We note that unlike the finite dimensional case, the continuity of $f$ is not automatic as it is shown in the following example:
\begin{Exa} Let  $\{ e_i\}$ be an algebraic basis of $X$ with $||e_i||=1,$ and let $f: X\longrightarrow \Bbb C $ be  the linear form satisfying $f(e_i)=i.$ Since $f$ is not bounded, $f$ is not holomorphic.
\end{Exa}
We now recall the definition of a  holomorphic polynomial.
\begin{Def} Let $\Delta_n:X\longrightarrow X^n$ be the mapping defined by $\Delta_n(x)=(x,\dots,x).$
A $n-$homogeneous holomorphic  polynomial $P:X \longrightarrow \Bbb C$ is the composition  of $\Delta_n$ with any   continuous $n-$linear map $L,$ that is, $P=L\circ\Delta_n.$ A holomorphic polynomial is a finite sum of homogeneous holomorphic polynomials.
\end{Def}
\begin{Def}
Let $H$ be a holomorphic function on $B(0,r), \ r>0.$ 
We say that $H$ is $n-$homogeneous if  $H(\lambda x) =\lambda^n H(x)$ for $\lambda \in \mathbb C, \ |\lambda| \le 1.$
\end{Def}
For the convenience of the reader, we give the proof of the next propositions.
\begin{Pro}
Let $f_n:B(0,r)\longrightarrow \Bbb C$ be holomorphic functions such that $\lim_{n \rightarrow \infty} f_n=f,$  uniformly on compact subsets of $B(0,r).$ Then $f$ is holomorphic on $B(0,r).$
\end{Pro}
\begin{proof} Since $f$ is clearly holomorphic when restricted to finite dimensional subspaces of $X,$ we only have  to show the continuity of $f.$ Let $\{z_n\}$ be a sequence in $X$ with $\lim_{n\rightarrow \infty} z_n=z\in X.$ By assumption, given  $\epsilon >0,$ there exists on the compact subset  $\{\{z_n\},z\},$  $M$ depending only on $\epsilon$ such that
\begin{equation}
 |f(z)-f_{M}(z)|+  |f_{M}(z_n)-f(z_n)|< \frac{2}{3}\epsilon.
\end{equation}
Hence, there exists $N>0$ such that  for $n\ge N$
\begin{equation}
|f(z)-f(z_n)| \le |f(z)-f_M(z)|+ |f_M(z)-f_M(z_n)|+ |f_M(z_n)-f(z_n)|< \epsilon.
\end{equation}
\end{proof}
\begin{Pro} Let $H:B(0,r)\longrightarrow \Bbb C$ be a $n-$homogeneous holomorphic function. Then $H$ is an entire function, that is a holomorphic function defined on $X.$
\end{Pro}
\begin{proof} For $z \in  X,$  we may write $z=\lambda_1 z_1, \  z_1 \in B(0,r).$
We define
$
\tilde H(z):= {\lambda _1}^n H(z_1)
$  and claim that it  is a well defined  entire function.
Indeed if   $z $ has another representation  $z=\lambda_2 z_2, \  z_2 \in B(0,r),$ then, without loss of generality, we will have $z_1=\dfrac{\lambda_2}{\lambda_1}z_2,\ \ |\dfrac{\lambda_2}{\lambda_1}|\le 1,$ which implies 
$H(z_1)=H(\dfrac{\lambda_2}{\lambda_1}z_2).$ 
\end{proof}
\begin{Pro}\label{hp}
Let $H:X\longrightarrow \Bbb C$ be a $n-$homogeneous holomorphic function. Then $H$ is a $n-$homogeneous  holomorphic polynomial. More precisely, there is a unique symetric  continuous $n-$linear map $L$ such that $H=L\circ\Delta_n.$
\end{Pro}
\begin{proof}
If $L$ exists, it has to satisfy
\begin{equation}
H(z_1x_1+\dots+z_nx_n)= L(\sum z_jx_j, \dots,\sum z_jx_j)=
\end{equation}
$${z_1}^n L(x_1, \dots,x_1)+ \dots +n!z_1z_2\dots z_nL(x_1, \dots, x_n)+ \dots + {z_n}^nL(x_n, \dots,x_n).
$$
Hence, we have the integral formula
\begin{equation}\label{hom}
L(x_1, \dots, x_n)=\dfrac{1}{n!}{(2\pi i)}^{-n}\int_{|z_1|=\epsilon_1}\dots \int_{|z_n|=\epsilon_n}  {\dfrac{H(z_1x_1+\dots+z_nx_n)   }{{z_1}^2\dots {z_n}^2}dz_1\dots dz_n}.
\end{equation}
Using \eqref{hom}, we obtain the uniqueness. 
For the existence, we use    \eqref{hom} as a definition  and show  that $L$ is  a (symetric)   $n-$linear map  using a linear change of variables, the continuity of $L$ beeing obtained by using $\epsilon_j$ small enough.
\end{proof}
\section{bounded holomorphic functions}
We start this section by recalling  the following definition:
\begin{Def}
Let $f:B(0,1)\longrightarrow \Bbb C$ be a holomorphic function. We define  $f_n(z)$  as
\begin{equation}\label{tay}
f_n(z):=\int_{0}^{2\pi}{f(e^{2\pi i t}z)e^{-2\pi in t}dt }
\end{equation}
\end{Def}
\begin{Lem}\label{tu} Let $f:B(0,1)\longrightarrow \Bbb C$ be a holomorphic function. Then  $f_n$   defined by \eqref{tay} is a  $n-$homogeneous holomorphic polynomial.
\end{Lem}
\begin{proof} We first notice that $f_n$  is a 
holomorphic fonction on the unit ball that is $n-$homogeneous along any finite dimensional subspace. We conclude then, using Proposition~\ref{hp}, that $f_n$ is a $n-$homogeneous holomorphic polynomial.
\end{proof}
\begin{Lem}\label{pin}Let $f:B(0,1)\longrightarrow \Bbb C$ be a holomorphic function. Then
\begin{equation}
f(z) = \sum _{n=1}^\infty{f_n(z)}
\end{equation}
\end{Lem}
\begin{proof}
The proof is achieved  by  restricting  to finite dimensional subspaces of $X.$
\end{proof}

The following theorem shows that uniform approximation by holomorphic polynomials is possible for bounded holomorphic functions. For the convenience of the reader, we give the proof.
\begin{Thm}\label{boundf}
Let $f:B(0,1)\longrightarrow \Bbb C$ be a  holomorphic function that is bounded on smaller balls $B(0,r), \ 0<r<1.$ Then  given   $\epsilon >0,$   there exists  a holomorphic polynomial  $h$  satisfying  $|f-h| < \epsilon$ on $B(0,r).$
\end{Thm}
\begin{proof}
We  claim that, for $0<\sigma<1,$ we have
\begin{equation}\label{ra}
\varlimsup _{n\rightarrow \infty} {\sigma ({|f_n|}_{B(0,1)}})^{\frac{1}{n}}< 1.
\end{equation}
Indeed, since $f$ is bounded on $B(0,1),$ we have
$
M\ge |f_n(\sigma z)|=\sigma ^n |f_n( z)|.
$ 
Hence 
$\sigma  {|f_n( z)|}^\frac{1}{n}<M^{\frac{1}{n}},$ which implies  $\varlimsup _{n\rightarrow \infty} {\sigma ({|f_n|}_{B(0,1)}})^{\frac{1}{n}}\le 1 .$ Choosing $\tilde \sigma<1$ satisfying   $\tilde \sigma = \lambda \sigma, \ \lambda>1, $  we obtain \eqref{ra}.
We then consider, for $r<1, $ the ball $B(0,r)=rB(0,1).$ Using \eqref{ra}, we obtain  that there exists $\mu <1$ such that $
r^n{|f_n|}_{B(0,1)}<\lambda ^n,$ for $n$ large enough. Using Lemma~\ref{pin}, we then  conclude that $f(z) = \sum _{n=1}^\infty{f_n(z)}$ uniformly on $B(0,r).$ The proof of the proposition is achieved using  Lemma~\ref{tu}.
\end{proof}
\section{A Counterexample}
In this section, we discuss the counterexample given by Lempert in \cite{Le08}. We start with the following definition:
\begin{Def} A set $S$ is called a bounding set if $||f||_S < \infty$ for every entire function.
\end{Def}
 In \cite{Di71}, Dineen shows that $l^{\infty}$ admits non compact closed bounding subsets. More precisely, he shows that any entire function on $l^{\infty}$ is bounded  on the set
 $
 S= \{e_n\},
$ where $e_n(j)=\delta_{jn}.$

Lempert considers   a sequence of   norms on $l^{\infty},$  all equivalent to the sup norm.  More precisely, he defines
\begin{equation}
|| z||_k=\dfrac{2}{k} sup_{j_1<j_2<\dots <j_k}|z(j_1)|+ |z(j_2)|+\dots +|z(j_k)|,\ \ z: \Bbb N:\longrightarrow \Bbb C \in l^{\infty}.   
\end{equation}
 Then he considers on $(l^{\infty},|| z||_k),$    the following  holomorphic function on the unit ball
\begin{equation}
f(z)=\sum_{j=1}^\infty jz(j)^j, 
\end{equation}
that is unbounded on the set $S= \{e_n\}.$ This shows, using Dineen's result, that the answer to  Problem (1) and  Problem (2) is negative  on $(l^{\infty},|| z||_k).$ We refer the interested reader to \cite{Le08} for the details.

\section{The general case}
In the light of the counterexample given by Lempert, we may ask the following question:
\begin{itemize}
\item Do Runge approximations (1) and (2) hold in  any separable Banach space?
\end{itemize}
Of particular interest is   the space  $C[0,1]$  since  every separable Banach space is isometric to a subspace  of $C[0,1].$    The Runge  approximations  are  still open for this space as well as for the Banach space $L^1(0,1).$
As said in the introduction, the obstruction to Runge approximations is Riesz Theorem. Therefore, one needs to look for "good" compact subsets that "replace" the unit ball.
The following lemma gives some  understanding of  what kind of sufficient conditions   is needed to obtain a positive answer to Problem (1) and Problem (2).

\begin{Lem}\label{chou}\cite{DS88} Let ${T_n}$ be a uniformly bounded sequence of linear operators in $X.$ If $\lim_n T_nx =x$ for every $x \in X,$ then this limit exists uniformly on any compact set. Conversely, if $\lim_n T_nx =x$ uniformly for $x$ in a bounded  set $K,$ and if, in addition, $ T_n K$ is relatively  compact for each $n,$ then $K$ relatively  is compact.
\end{Lem}
We recall the following definitions.
\begin{Def}
A series $\sum_{n=1}^{\infty} x_n$ is said to converge unconditionally if $\sum_{n=1}^{\infty} x_{\pi (n)}$ converges for every permutation $\pi$ of the integers.
\end{Def}
\begin{Def}
Let $X$ be a complex Banach space. A sequence $\{X_n    \}$ of closed subspaces of $X$ is called a Schauder decomposition of $X$ if every $x\in X$ has a unique representation of the form 
$x=\sum_{n=1}^{\infty} x_n,$ with $x_n \in X.$ The Schauder decomposition is unconditional if for every $x\in X,$ the series $\sum_{n=1}^{\infty} x_n$ which represents $x$ converges unconditionally.
\end{Def}
\begin{Rem}\cite{IS81} A decomposition  $\{X_n    \}$  of a Banach space $X$ is a Schauder decomposition of $X$ if and only if the projections $\{ P_j\}$ defined by
\begin{equation}
P_j\sum_{n=1}^{\infty} x_n= \sum_{n=1}^{j} x_n
\end{equation} are continuous. Moreover $\{ P_j\}$ is a uniformly bounded sequence of linear operators in $X.$
\end{Rem}
\begin{Exa} 

$l^p(\Bbb N)= \{z: \Bbb N \longrightarrow \Bbb C, \ ||z|| = {(\sum _{j=1}^\infty {|z(j) |}^p)}^{\frac{1}{p}}< \infty         \}  $  admits an unconditional Schauder basis, i.e. $\dim X_n=1.$
\end{Exa}
\begin{Exa} Let $X$ be the space  of  compact operators on $l^2(\Bbb N)$ which have a triangular representing matrix with respect to  $
 \{e_n\},
$ where $e_n(j)=\delta_{jn}.$ Let $X_n$  be the subspace of $X$  defined as  $$ X_n= \{ T \in X \ |\  Te_j =0, \ j \ne n\}.$$
Then $\{ X_n \}$ is an unconditional Schauder  decomposition of $X$ with $\dim X_n< \infty.$ In this case, we say that $X$ has a UFDD (unconditional finite dimensional decomposition). Moreover, by a result of Gordon and Lewis, $X$ does not have an unconditional Schauder basis. See \cite{Go74}, \cite{Li77}.
\end{Exa}
\begin{Exa}\label{felix}
It is well known that the  space  of  compact operators on $l^2(\Bbb N)$ has an unconditional Schauder decomposition into UFDD, but has no UFDD itself. See \cite{Kw70}.
\end{Exa}
 Assume now that $X$ admits  an unconditional   Schauder decomposition ${\{X_n\}}_{n=1}^{\infty}.$  Let   $x= \sum_{n=1}^{\infty}x_n,$   $x_n \in X_n,$  be  the unique representation of $x.$ It is known that  for every sequence of complex numbers $\theta=\{\theta_n\}, \ |\theta_n |\le 1,n \in \mathbb N,$ the operator $M_{\theta}$ defined by  $$M_{\theta} \sum_{n=1}^{\infty}x_n= \sum_{n=1}^{\infty}\theta_nx_n$$ is a bounded linear operator. The (finite) constant sup$_{\theta}\|M_{\theta}\|$ is called the {\it unconditional constant} of the decomposition.(See \cite{Li77} and \cite{IS81} for  details). It is clear  that one can always define on $X$ an equivalent norm ${\| \|}_1$  so that the unconditional constant becomes 1. (Take ${\| x\|}_1=$sup$_{\theta}\|M_{\theta}x\|$). In other words, we have

\begin {equation}\label{unc} \|\sum_{n=1}^{\infty}\theta_n {x_n\|}_1 \le \|\sum_{n=1}^{\infty}{x_n\|}_1, \ \  |\theta_n |\le 1, \ n \in \mathbb N. 
\end{equation}
 
Using   Lemma~\ref{chou} and \eqref{unc}, one proves

\begin{Pro}\label{com}
Let $X$ be a complex Banach space admitting an unconditional Schauder decomposition  $\{X_n    \}$  with  unconditional constant  one.  Then the following holds.
\begin{enumerate}
\item { $M_{\theta}(B(0, R))$ is relatively compact in $B(0,R)$ for any sequence  of complex numbers $\theta=\{\theta_n\}, \ |\theta_n |< 1,$ which converges to $0$ if and only if $\dim X_n< \infty,$ for any $n.$}
\item {For any compact $K \subset B(0,R),$ there exists a  sequence of complex numbers  $\theta=\{\theta_n\}, \ |\theta_n |< 1$ which converges to $0,$ and a compact $L \subset B(0,R)$ so that  $M_{\theta}L = K.$ }
\end{enumerate} 
\end{Pro}
\begin{proof}
 We note  that  $P_j(M_{\theta}(B(0,R)))$  is relatively  compact if and only if  $\dim X_n< \infty,$ for any $n.$
 The rest of the proof is similar to the one of Proposition 1.2 in \cite{Le00b}.
\end{proof}
\begin{Rem}\label{comp} The sets  $M_{\theta}(B(0, R))$ in 
 Proposition~\ref{com} will play the role of "`good"'  relatively compact sets that replace the unit ball in the case of a space admitting  a UFDD.
\end{Rem}

\section{The case of Banach spaces admitting a UFDD}
In the light of  the proof of   Theorem~\ref{boundf}, Remark~\ref{comp} leads to the following  definitions that appear in  \cite{Le99}, \cite{Le00b}, \cite{Jo04} and \cite{Me06}.

Let  $X$  be a complex Banach space admitting an unconditional Schauder decomposition  $\{X_n    \}$  with  unconditional constant  one.  Let $z \in X$  be  given by  its unique representation  $z= \sum_{n=1}^{\infty}z_n,$ with $z_n \in X_n$ for every $n.$
For $m\in \mathbb N,$ one defines   
\begin{equation}T(m):=\{e^{2\pi i t},\  {t}=\{t_n\},\   t_n\in \mathbb R, \  t_n=0 \ \text{for} \ n>[\ \root \of  m\ ]\}\end{equation}
and
\begin{equation}K(m):=\{\  {k}=\{k_n\},\  k_n\in \mathbb N \cup \{0\},\ k_n=0 \ \text{for}\  n>[\ \root \of  m\ ], \  \sum_{n=1}^{[\ \root \of  m]}\  k_n \le m\} \end{equation}
\begin{Def}Let $f$ be a holomorphic function on $B(0,R).$
 For     $k \in K(m)$ and $e^{2\pi i s}\in T(m),$ we define $f^{k}(z)$ as
\begin{equation}f^{k}(z):=\int_{T(m)}f(M_{e^{2\pi i s}}z)e^{-2\pi ik.s}ds,\end{equation} where   $ds$ is the normalized  Haar measure  on $T(m)$.
\end{Def}
 Note that $f^k$ is  homogeneous of degree $k_n$ in $z_n.$  
\begin{Def} The formal series associated to $f$  given by \begin{equation}
\sum_{m=0}^{\infty} \ \sum_{k\in K(m)}{(f_m)}^k(z)
\end{equation} is called the Josefson series.
\end{Def}
We have the following proposition:
\begin{Pro}\cite{Me06}\label{tou}
Let  $X$  be a complex Banach space admitting an unconditional Schauder decomposition  $\{X_n    \}$  with  unconditional constant  one.
Then the Josefson series converges  to $f,$ uniformly on compact sets of $B(0,R)$  if $\dim X_n < \infty$ for any n.
\end{Pro}
\begin{Rem} In  Proposition~\ref{tou}, the  proof of the uniform convergence on compact sets  relies on the fact that the sets  $M_{\theta}(B(0, R))$ are compact. It would be interesting to know if it holds without the assumption on the dimension of each $X_n.$
\end{Rem}

Making use of the Josefson series, one can show that problem (2) holds in spaces admitting  an unconditional Schauder decomposition  $\{X_n    \}$  
with  $\dim X_n < \infty$ for any n. More precisely, we have
\begin{Thm}\cite{Me06}
 Let  $X$  be a complex Banach space admitting  an unconditional  Schauder decomposition  $\{X_n    \}$  
with  $\dim X_n < \infty$ for any n.  Then there exists an equivalent norm  $\| \|$  on $X$  for which the Runge approximation problem (1) holds.
\end{Thm}
We  may ask the following question
\begin{itemize}
\item  Does the Runge approximation problem (2) hold in  complex Banach spaces admitting  an unconditional  Schauder decomposition into UFDD? 
\end{itemize}
Note that the  "model" case of such a space is given by example \eqref{felix}.
\begin{Rem} It is known that the spaces $C[0,1]$ and $L^1[0,1]$ do not admit any unconditional Schauder decomposition into UFDD. Indeed, Lindenstrauss and Pelczynski in \cite {LP71}  showed that if $(X_n)$  is an unconditional Schauder decomposition of  $C[0,1],$ then at least one $X_n \approx C[0,1] ,  $ which is impossible \cite{IS81}.  N.J. Kalton in \cite{Ka78} showed that the same holds for $L^1[0,1].$
\end{Rem}

\end{document}